\documentclass[a4paper,11pt,reqno]{amsart}
\usepackage[T1]{fontenc}
\usepackage[utf8]{inputenc}
\usepackage{lmodern}
\usepackage{amsmath, amsthm, amssymb,amscd, mathrsfs, amsfonts, mathtools}

\usepackage[colorlinks]{hyperref}

\usepackage{euler}
\usepackage{times}
\usepackage[all]{xy}
\usepackage{todonotes}
\usepackage{xcolor}

\newtheorem{theorem}{Theorem}[section]
\newtheorem{question}[theorem]{Question}

\usepackage[lite]{amsrefs}

\renewcommand{\PrintDOI}[1]{\href{http://dx.doi.org/\detokenize{#1}}{doi: \detokenize{#1}}%
  \IfEmptyBibField{pages}{, (to appear in print)}{}}

\def\commutatif{\ar@{}[rd]|{\circlearrowleft}}

\newcommand{\eq}[1][r]
   {\ar@<-3pt>@{-}[#1]
    \ar@<-1pt>@{}[#1]|<{}="gauche"
    \ar@<+0pt>@{}[#1]|-{}="milieu"
    \ar@<+1pt>@{}[#1]|>{}="droite"
    \ar@/^2pt/@{-}"gauche";"milieu"
    \ar@/_2pt/@{-}"milieu";"droite"}

\def\dar[#1]{\ar@<2pt>[#1]\ar@<-2pt>[#1]}
  \entrymodifiers={!!<0pt,0.7ex>+} %%%%%%%%%%%%%%%%%

\newcommand{\bigon}[4][r]{% %%%%% bigons
    \ar@/^1pc/[#1]^{#2}_*=<0.3pt>{}="HAUT"
    \ar@/_1pc/[#1]_{#3}^*=<0.3pt>{}="BAS"
    \ar@{=>} "HAUT";"BAS" ^{#4}
  }

\newcommand{\bigons}[6][r]{  %%%%% Vertical composition of bigons
    \ar@/^2pc/[#1]^{#2}_*=<0.3pt>{}="HAUT"
    \ar@{}    [#1]     ^*=<0.3pt>{}="MILIEUHAUT"
                       _*=<0.3pt>{}="MILIEUBAS"
    \ar[#1]_(0.3){#3}
    \ar@/_2pc/[#1]_{#4}^*=<0.3pt>{}="BAS"
    \ar@{=>} "HAUT";"MILIEUHAUT" ^{#5}
    \ar@{=>} "MILIEUBAS";"BAS" ^{#6}
  }

\newcommand\rTo{\longrightarrow}

\newcommand\mto{\longmapsto}

\theoremstyle{definition}

\newtheorem{lemma}[theorem]{Lemma}
\newtheorem{corollary}[theorem]{Corollary}

\newtheorem{conjecture}[theorem]{Conjecture}

\theoremstyle{definition}

\theoremstyle{remark}
\newtheorem{remark}[theorem]{Remark}
\numberwithin{equation}{section}

\title{On the Classification of Topological Quandles}

\author{Zhiyun Cheng}
\address{School of Mathematical Sciences,
Beijing Normal University, Laboratory of Mathematics and Complex Systems, Ministry of Education, Beijing 100875, China}
\email{czy@bnu.edu.cn}

\author{Mohamed Elhamdadi}
\address{Department of Mathematics,
University of South Florida, Tampa, FL 33620, U.S.A.}
\email{emohamed@math.usf.edu}

\author{Boris Shekhtman}
\address{Department of Mathematics,
    University of South Florida, Tampa, FL 33620, U.S.A.}
\email{shekhtma@usf.edu}
\begin{document}
\maketitle
\begin{abstract}
We investigate the classification of topological quandles on some simple manifolds.  Precisely we classify all Alexander quandle structures, up to isomorphism, on the real line and the unit circle. For the closed unit interval $[0, 1]$, we conjecture that there exists only one topological quandle structure on it, i.e. the trivial one. Some evidences are provided to support our conjecture.
\end{abstract}
\section{Introduction}
Quandles and racks are in general non-associative structures whose axioms correspond to the algebraic distillation of the Reidemeister moves in knot theory. Quandles have been investigated by topologists for the purpose of constructing knot and link invariants, and by algebraists for the aim of classification of finite quandles \cite{Hulpke} and the construction of pointed Hopf algebras \cite{AG}. The earliest known work on racks (see definition below) is contained
in the 1959 correspondence between John Conway and Gavin Wraith who studied racks in the context of the conjugation operation in a group.   Around 1982,  Joyce \cite{Joyce} (used the term quandle) and Matveev \cite{Matveev} (who called them distributive groupoids)  introduced independently the notion of quandle.

A \emph{topological rack} $X$ consists of both a rack structure and a topological structure which are compatible \cite{Elhamdadi-Moutuou:Foundations}.  More precisely, the rack binary operation $*:  X\times X  \rTo X$ sending $ (x,y) \mto  x * y$ is continuous with respect to the topological structure, the right multiplication $R_x: y\mto y * x$ is a homeomorphism, and also the binary operation satisfies the {\it right distributivity}:  $\forall x,y,z\in X$,  $(x * y) * z=(x * z)* (y * z)$.  Here $X\times X$ is viewed as a topological space with the product topology. See \cite{EN, Joyce, Matveev} for more details.

In \cites{Rubinsztein:Top_Quandles}, Rubinsztein investigated topological quandles and extended the notion of coloring of a knot or link by a quandle to include topological quandles.  He proved that the coloring space of the link is a topological space (defined up to a homeomorphism) that is an invariant of the link and gave several computational examples. Jacobsson and Rubinsztein \cite{Jacobsson-Rubinsztein} computed the space of colorings of all prime knots with up to seven crossings and of all $(2,n)$-torus links. They also observed some similarities between the space of colorings of knots and Khovanov homology for all prime knots with up to seven crossings and for some eight-crossing knots. See \cite{KM, Jacobsson-Rubinsztein} for more details.

The aim of this paper is to investigate the topological quandle structures on the real line and the interval. In Section 2, we give a brief review of the definition of topological quandles and list some examples of topological quandles. Section 3 is devoted to the classification of all Alexander quandles on the Euclidean spaces $\mathbb{R}^1$ and $S^1$. In Section 4 we discuss the homogenous topological quandles. Some open problems are listed in Section 5. In particular, we conjecture that there exists only one topological quandle structure on the interval $[0, 1]$. Some evidences are also provided to support our conjecture.

\section{Basic Review of Topological Quandles}\label{sec2}
A \emph{topological rack} $X$ is a topological space with a rack binary operation $*:  X\times X  \rTo X$ sending $ (x,y) \mto  x * y$ that is continuous with respect to the topological structure, such that the right multiplication $R_x: y\mto y\ast x$ is a homeomorphism, and also the binary operation satisfies the {\it right distributivity}:  $\forall x,y,z\in X$,  $(x * y) * z=(x * z)* (y * z)$. In particular, if $x\ast x=x$ for each $x\in X$, then we say that $X$ is a \emph{topological quandle}. It is not difficult to observe that $R_x^{-1}: y\mto y\ast^{-1}x$ also provides $X$ a topological quandle structure $(X, \ast^{-1})$, called the \emph{dual quandle} of $(X, \ast)$. The set $Aut(X)$ of quandle automorphisms of $X$ forms a group under composition. Furthermore, if $X$ is a locally compact, locally path-connected, Hausdorff topological space, when equipped with the compact-open topology, $Aut(X)$ is a topological group \cite{Arens}. Recall from proposition 3.1 of \cite{Elhamdadi-Moutuou:Foundations} that the \emph{inner representation} of $X$ is the continuous map \[
 \begin{array}{lccc}
 R: & X & \rTo &Aut(X)\\
 & x & \mto & R_x,
 \end{array}
 \]
and that the {\it inner automorphism group} $Inn(X)$ of $X$ is the closure of the subgroup generated by the image of $X$ by $R$ in $Aut(X)$, $Inn(X):= \overline{<R(X)>}\subset Aut(X)$. Note that $Inn(X)$ is a normal subgroup of $Aut(x)$, since $fR_x=R_{f(x)}f$ for any automorphism $f\in Aut(X)$.

A topological quandle $X$ is called \emph{homogeneous} if  the automorphism group $Aut(X)$ acts transitively on $X$. If the group $Inn(X)$ acts transitively on $X$ then we call it {\it algebraically connected} or \emph{indecomposable} (so there will be no confusion with topological connectedness). In other words, for any $x, y\in X$ there are $\{x_1, \cdots, x_n\}\subseteq X$ such that
\begin{center}
$(\cdots((x\ast ^{\varepsilon_1}x_1)\ast^{\varepsilon_2}x_2)\cdots)\ast^{\varepsilon_n}x_n=y$,
\end{center}
where $\varepsilon_i=\pm1$. Obviously, an algebraically connected topological quandle is a homogeneous topological quandle.

Suppose we are given a topological quandle $X$, then we can consider the algebraic connectedness and the topological connectedness. In general, there is no general relation between them. For example, any algebraically connected quandle which contains more than one element equipped with the discrete topology is algebraically connected but not topologically connected. On the other hand, any connected topological space which contains more than one point equipped with the trivial quandle structure (i.e. $R_x=id$ for all $x$) is topologically connected but not algebraically connected. However we have the following result, which can be regarded as a extension of the obvious fact that if a quandle contains only two elements then it is algebraically disconnected (actually it must be trivial).
\begin{lemma}
Let $X$ be a topological quandle which consists of two topologically path-connected components, then $X$ can not be algebraically connected.
\end{lemma}
\begin{proof}
Assume that $X$ consists of two connected components $X_1$ and $X_2$. For any $x_1\in X_1$ and $x_2\in X_2$, since $x_1\ast x_1=x_1, x_2\ast x_2=x_2$, it follows that
$X_1\ast x_1=X_1$ and $X_2\ast x_2=X_2$. If for some $x_1\in X_1$ and $x_2\in X_2$, $x_1\ast x_2\in X_2$. Choose a path $l$ connecting $x_1\ast x_2$ and $x_2\ast x_2=x_2$ in $X_2$. Since $R_{x_2}$ is an automorphism of $X$, then $R^{-1}_{x_2}(l)$ is a path connecting $x_1$ and $x_2$. This contradicts the assumption that $X_1$ and $X_2$ are two connected components of $X$. It follows that for any $x_1\in X_1$ and $x_2\in X_2$, $x_1\ast x_2\in X_1$, which means $X_1\ast x_2=X_1$. We conclude that for any $x\in X$ we have $X_1\ast^{\pm1} x=X_1$, it follows that $X$ is algebraically disconnected.
\end{proof}

We end this section with some examples of topological quandles. Obviously every quandle can be trivially made into a topological quandle by considering it with the discrete topology or the indiscrete topology. On the other hand, for each topological space $X$, one can define a quandle structure on $X$ by defining $x\ast y=x$ for any $x, y\in X$, which is a trivial quandle. In this paper we are mainly interested in the existence of nontrivial rack/quandle structures on topological spaces.

If $X$ is a topological group, then we can associate two quandle structures on $X$, the \emph{conjugation quandle Conj$(X)$} and the \emph{core quandle Core$(X)$}. The operations of these two quandles are defined by $x\ast y=yxy^{-1}$ and $x\ast y=yx^{-1}y$ respectively. Note that when $X$ is an abelian group then the associated conjugation quandle is trivial. However the core quandle is a trivial quandle if and only if $X$ is an abelian group with every nontrivial element of order 2. Actually, if the core quandle is trivial, then $x\ast y=yx^{-1}y=x$. By putting $x=1$ one obtains that $y^2=1$ for any element of $X$. Now $xy=yx$ follows immediately from $(xy)^2=1$. Conversely, if $X$ is an abelian group with every nontrivial element of order 2, then $x\ast y=yx^{-1}y=y^2x=x$, which implies that the quandle structure is trivial. Moreover, if we have a homeomorphism $\sigma$ of $X$, the operation $x\ast y=\sigma(xy^{-1})y$ makes $X$ into a topological quandle.

On the other hand, note that for a product space $X\times Y$, if $(X, \ast_X)$ and $(Y, \ast_Y)$ are both topological quandles, then $X\times Y$ is a topological quandle with operation $(x_1, y_1)\ast(x_2, y_2)=(x_1\ast_X x_2, y_1\ast_Y y_2)$.

Here we list some familiar examples of topological rack/quandle, more examples can be found in \cites{Rubinsztein:Top_Quandles,Elhamdadi-Moutuou:Foundations}.
\begin{enumerate}
  \item
   \emph{The real line} $\mathbb{R}^1$: Since $\mathbb{R}^1$ is a topological group, as we mentioned above, each homeomorphism of $\mathbb{R}^1$ makes $\mathbb{R}^1$ into a topological quandle. Recall that a homeomorphism $\sigma$ of $\mathbb{R}^1$ $($as a topological group$)$  has the form $\sigma(x)=tx$ $(t\neq0)$, which induces the Alexander quandle structure on $\mathbb{R}^1$ with operation $x\ast y=tx+(1-t)y$ $(t\neq0)$. Similarly the operation $(x_1, \cdots, x_n)\ast(y_1, \cdots, y_n)=(t_1x_1+(1-t_1)y_1, \cdots, t_nx_n+(1-t_n)y_n)$ $(t_i\neq0, 1\leq i\leq n)$ derived from the homeomorphism of $\mathbb{R}^n$ makes $\mathbb{R}^n$ into a topological quandle.

  \item
   \emph{The sphere} $\mathbb{S}^n$: Consider the unit sphere in $\mathbb{R}^{n+1}$, define $x\ast y=2(x\cdot y)y-x$, here $\cdot$ denotes the inner product of $\mathbb{R}^{n+1}$. It is easy to show that this operation makes $S^n$ into a topological quandle.

  \item
  \emph{The projective space} $\mathbb{RP}^n$: A quandle structure of $\mathbb{RP}^n$ can be directly derived from the quandle structure on $\mathbb{S}^n$ mentioned above.

  \item
   \emph{The Grassmannian} $Gr(r, V)$: Let $Gr(r, V)$ be the Grassmannian of $r$-dimensional linear subspaces a vector space $V$. For two subspaces $U, W\in Gr(r, V)$ and $v\in V$ we define
  \begin{center}
  $v\ast W=2\sum\limits_{i=1}^r(w_i\cdot v)w_i-v$,
  \end{center}
  where $\{w_i\}$ denotes an orthonormal basis of $W$. Now $U\ast W$ is defined to be $\{u\ast W|u\in U\}$. It is easy to check that this operation induces a topological quandle structure on $Gr(r, V)$.
\end{enumerate}

Given two topological quandles $(X, \ast_X)$ and $(Y, \ast_Y)$, we say that $(X, \ast_X)$ and $(Y, \ast_Y)$ are \emph{isomorphic} if there exists a homeomorphism $f$ from $X$ to $Y$ such that $f(x_1\ast_X x_2)=f(x_1)\ast_Y f(x_2)$. If $X$ and $Y$ are oriented topological spaces then we require that $f$ is orientation-preserving. Recall that for each element $x$ of an unoriented topological quandle $X$, the right multiplication $R_x: X\rightarrow X$ is a homeomorphism. We remark that if $X$ is a connected oriented topological space, then either $R_x$ $(x\in X)$ are all orientation-preserving or all orientation-reversing.

We give one simple example to show that when one places different topologies on the same quandle, it is possible to obtain two different topological quandles. Consider the quandle $X$ which consists of three elements $\{1, 2, 3\}$, and the quandle operations are defined below
\begin{center}
$\begin{bmatrix}
  1 & 1 & 1 \\
  3 & 2 & 2 \\
  2 & 3 & 3 \\
\end{bmatrix}$.
\end{center}
Here the $(i, j)$ entry denotes $i\ast j$ $(1\leq i, j\leq3)$. Let $\tau_1=\{\varnothing, \{1\}, \{1, 2, 3\}\}$ and $\tau_2=\{\varnothing, \{1\}, \{2\}, \{1, 2\},\{1, 2, 3\}\}$ be two topologies on $X$. It is easy to see that $(X, \tau_1)$ and $(X, \tau_2)$ are both topological quandles. However they are not isomorphic as topological quandles since they are not even homeomorphic as topological spaces.

In this article, we want to investigate the following problem: given a topological space $X$, how many different (up to isomorphism) topological quandle structures are there? In particular, is there a topological space which can only be equipped with the trivial quandle structure?

\section{Classification of Topological Affine Quandles on $\mathbb{R}$}\label{sec3}
Let $(\mathbb{R}, \ast_1), (\mathbb{R}, \ast_2)$ be two topological quandles, where the operations are defined by $x\ast_1 y=t_1x+(1-t_1)y$ and $x\ast_2 y=t_2x+(1-t_2)y$ $(t_i\neq0, i\in\{1, 2\})$. The aim of this section is to determine when topological quandles $(\mathbb{R}, \ast_1)$ and $(\mathbb{R}, \ast_2)$ are isomorphic. We have the following lemmas.

\begin{lemma}\label{1}
If $t_1>0$ and $t_2<0$, then $(\mathbb{R}, \ast_1)$ and $(\mathbb{R}, \ast_2)$ are different topological quandles.
\end{lemma}
\begin{proof}
If not, suppose $f: \mathbb{R}\rightarrow\mathbb{R}$ induces an isomorphism between $(\mathbb{R}, \ast_1)$ and $(\mathbb{R}, \ast_2)$. Notice that $R_x$ in $(\mathbb{R}, \ast_1)$ is orientation-preserving, however $R_{f(x)}$ in $(\mathbb{R}, \ast_2)$ is orientation-reversing.
\end{proof}

Without loss of generality, let us assume that $t_1$ and $t_2$ are both positive.

\begin{lemma}\label{2}
If $t_1=1$ and $t_2\neq 1$, then $(\mathbb{R}, \ast_1)$ and $(\mathbb{R}, \ast_2)$ are different topological quandles.
\end{lemma}
\begin{proof}
If there exists a homeomorphism $f$ of $\mathbb{R}$ which induces an isomorphism on the quandle structure, then
\begin{center}
$f(x)=f(x\ast_1y)=f(x)\ast_2f(y)=t_2f(x)+(1-t_2)f(y)$,
\end{center}
which implies $f$ is a constant function. This contradicts with the assumption that $f$ is an isomorphism.
\end{proof}

\begin{lemma}\label{3}
If $t_1>1 $ and $0<t_2<1$, then $(\mathbb{R}, \ast_1)$ and $(\mathbb{R}, \ast_2)$ are different topological quandles.
\end{lemma}
\begin{proof}
We assume there exists a homeomorphism $f$ of $\mathbb{R}$ which preserves the quandle structure. It follows that
\begin{center}
$f(t_1x+(1-t_1)y)=f(x\ast_1y)=f(x)\ast_2f(y)=t_2f(x)+(1-t_2)f(y)$.
\end{center}
Note that $f(x)+b$ also gives an isomorphism from $(\mathbb{R}, \ast_1)$ to $(\mathbb{R}, \ast_2)$, without loss of generality we assume that $f(0)=0$. Let $y=0$, we obtain
$f(t_1x)=t_2f(x)$. However this contradicts with the assumption that $f$ is monotonic.
\end{proof}

\begin{lemma}\label{4}
If $t_1> t_2> 1$ then $(\mathbb{R}, \ast_1)$ and $(\mathbb{R}, \ast_2)$ are different topological quandles.
\end{lemma}
\begin{proof}
If $\phi$ is a quandle isomorphism between $(\mathbb{R}, \ast_1)$ and $(\mathbb{R}, \ast_2)$, then $\phi$ is a homeomorphism of the real line satisfying $\phi(t_1x+(1-t_1)y)=t_2\phi(x)+(1-t_2)\phi(y)$.  Without loss of generality we can assume $\phi(0)=0$. Since $\phi(0)=0$ thus $\phi(1)\neq 0$.  By considering the function $\psi(x)=\frac{\phi(x)}{\phi(1)}$ which is still a quandle isomorphism, one can then assume that $\phi(1)=1$. Then for all $x \in \mathbb{R}$, we have
\begin{center}
$\phi(t_1x)=t_2\phi(x)$ and $\phi((1-t_1)x)=(1-t_2)\phi(x).$
\end{center}
In other words,
\begin{center}
$\frac{1}{t_2}\phi(x)=\phi(\frac{x}{t_1})$ and $\frac{1}{1-t_2}\phi(x)=\phi(\frac{x}{1-t_1})$.
\end{center}
We conclude that for any $m, n\in\mathbb{Z}$ we have
\begin{center}
$\phi(\frac{t_1^m}{(1-t_1)^{2n}}x)=\frac{t_2^m}{(1-t_2)^{2n}}\phi(x)$,
\end{center}
setting $x=1$ yields
\begin{center}
$\phi(\frac{t_1^m}{(1-t_1)^{2n}})=\frac{t_2^m}{(1-t_2)^{2n}}$.
\end{center}
Assume $\frac{ln(1-t_1)^2}{ln(t_1)}$ is a irrational number, where $ln$ stands for the natural logarithm. Now we can choose a sequence $\{\frac{m_i}{n_i}\}_{i\in\mathbb{N}}$ which converges to $\frac{ln(1-t_1)^2}{ln(t_1)}$ (if $\frac{ln(1-t_1)^2}{ln(t_1)}$ equals a rational number $\frac{m}{n}$, we just choose $\frac{m_i}{n_i}=\frac{m}{n}$ for any $i\in\mathbb{N}$), which means that $\phi(\frac{t_1^{m_i}}{(1-t_1)^{2n_i}})$ converges to $\phi(1)=1$. On the righthand side we have
\begin{center}
$\lim\limits_{i\rightarrow\infty}\frac{t_2^{m_i}}{(1-t_2)^{2n_i}}=\lim\limits_{i\rightarrow\infty}\exp(m_iln(t_2)-n_iln(1-t_2)^2)$.
\end{center}
In order to obtain the contradiction, it suffices to show that
\begin{center}
$\lim\limits_{i\rightarrow\infty}(m_iln(t_2)-n_iln(1-t_2)^2)\neq0$.
\end{center}
One computes
\begin{flalign*}
&\lim\limits_{i\rightarrow\infty}(m_iln(t_2)-n_iln(1-t_2)^2)&\\
=&\lim\limits_{i\rightarrow\infty}(\frac{m_i}{n_i}ln(t_2)n_i-n_iln(1-t_2)^2)\\
=&\lim\limits_{i\rightarrow\infty}n_i(\frac{ln(t_1-1)^2ln(t_2)-ln(t_2-1)^2ln(t_1)}{ln(t_1)})\\
\neq&0.
\end{flalign*}
The last inequality follows from the fact that $\lim\limits_{i\rightarrow\infty}n_i\neq0$ and
\begin{center}
$ln(t_1-1)^2ln(t_2)-ln(t_2-1)^2ln(t_1)\neq0$,
\end{center}
this can be proved by checking that the function $\frac{ln(x-1)^2}{ln(x)}$ is a monotonous increasing function on $(1, +\infty)$.
\end{proof}

\begin{lemma}\label{5}
If $0<t_2<t_1<1$, then $(\mathbb{R}, \ast_1)$ and $(\mathbb{R}, \ast_2)$ are different topological quandles.
\end{lemma}
\begin{proof}
Suppose $\phi$ is a quandle isomorphism between $(\mathbb{R}, \ast_1)$ and $(\mathbb{R}, \ast_2)$. As before, we assume that $\phi(0)=0$ and $\phi(1)=1$, therefore $\phi(x)$ is a monotonous increasing function on $\mathbb{R}$. Similar to the proof of Lemma \ref{4}, we have
\begin{center}
$\phi(\frac{t_1^m}{(1-t_1)^{n}})=\frac{t_2^m}{(1-t_2)^{n}}$ $(m, n\in\mathbb{Z})$.
\end{center}
Since $\frac{ln(1-x)}{ln(x)}$ is a monotonous increasing function on $(0, 1)$, it follows that $\frac{ln(1-t_2)}{ln(t_2)}<\frac{ln(1-t_1)}{ln(t_1)}$. Choose a pair of positive integers $m, n$ such that
\begin{center}
$\frac{ln(1-t_2)}{ln(t_2)}<\frac{m}{n}<\frac{ln(1-t_1)}{ln(t_1)}$.
\end{center}
Now we have $\frac{t_1^m}{(1-t_1)^{n}}<1$ but $\frac{t_2^m}{(1-t_2)^{n}}>1$, which contradicts with the fact that $\phi(x)$ is a monotonous increasing function.
\end{proof}

To sum up, the theorem below follows directly from Lemma \ref{1}-\ref{5}.
\begin{theorem}\label{6}
Let $t_1 $ and $t_2 $ be two distinct real numbers both distinct from zero.  Then the Alexander quandle structures $(\mathbb{R}, \ast_1)$ and $(\mathbb{R}, \ast_2)$ can not be isomorphic.
\end{theorem}

\begin{remark}
Recently, Theorem \ref{6} was generalized from $\mathbb{R}$ to $\mathbb{R}^n$ (see \cite{ESZ} for more details). Let consider the two topological quandles $(\mathbb{R}^n, \ast_t)$ and $(\mathbb{R}^n, \ast_s)$, where the quandle operations are defined by $x\ast_t y=tx+(I_{n\times n}-t)y$ and $x\ast_s y=sx+(I_{n\times n}-s)y$ respectively. Here we use the following notations
\begin{center}
$x=\begin{pmatrix}
     x_1 \\
     \vdots\\
     x_n\\
   \end{pmatrix}
, y=\begin{pmatrix}
     y_1 \\
     \vdots\\
     y_n\\
   \end{pmatrix}
, t=\begin{pmatrix}
      t_1 &  & \\
       & \ddots&  \\
      &  & t_n \\
    \end{pmatrix}
, s=\begin{pmatrix}
      s_1 &  & \\
       & \ddots&  \\
      &  & s_n \\
    \end{pmatrix}$,
\end{center}
where $t$ and $s$ are both diagonal matrices and $t_i\neq0$ and $s_i\neq0$ for all $1\leq i\leq n$. It was proved in \cite{ESZ} that $(\mathbb{R}^n, \ast_t)$ and $(\mathbb{R}^n, \ast_s)$ are isomorphic if and only if there exists a matrix $f\in GL_n(\mathbb{R})$ such that $ftf^{-1}=s$.
\end{remark}

\begin{corollary}
Let $S^1=\{e^{i\theta}|0\leqslant\theta\leqslant 2\pi\}$, and $\ast_i$ $(i=1, 2)$ be two quandle operations on $S^1$ which are defined as $e^{i\theta_1}\ast_ie^{i\theta_2}=e^{i(t_i\theta_1+(1-t_i)\theta_2)}$ $(0<t_i\leq 1, i=1, 2)$. Then $(S^1, \ast_1)$ and $(S^1, \ast_2)$ are isomorphic if and only if $t_1=t_2$.
\end{corollary}
\begin{proof}
The main idea of the proof is similar to the the proof of Lemma \ref{5}. If there exists an isomorphism $\phi$ between $(S^1, \ast_1)$ and $(S^1, \ast_2)$, we can assume that $\phi(0)=0$ and $\phi(2\pi)=2\pi$. Then $\phi$ induces an homeomorphism $f$ from $[0, 2\pi]$ to $[0, 2\pi]$ which satisfies $f(0)=0, f(2\pi)=2\pi$ and $f(t_1\theta_1+(1-t_1)\theta_2)=t_2f(\theta_1)+(1-t_2)f(\theta_2)$ for any $0\leq\theta_i\leq2\pi$ $(i=1, 2)$. In particular, we have
\begin{center}
$f(2t_1^m\pi)=2t_2^m\pi$ and $f(2(1-t_1)^n\pi)=2(1-t_2)^n\pi$
\end{center}
for any $m, n\in\mathbb{Z}^+$.

Notice that if $t_1=1$, then it follows immediately that $t_2$ must be 1. Without loss of generality, let us assume $0<t_2<t_1<1$. Since $\frac{ln(1-x)}{ln(x)}$ is a monotonous increasing function on $(0, 1)$, we can find two positive integers $m, n$ such that
\begin{center}
$\frac{\ln(1-t_2)}{\ln(t_2)}<\frac{m}{n}<\frac{\ln(1-t_1)}{\ln(t_1)}$.
\end{center}
It follows that $(t_2)^m<(1-t_2)^n$ but $(t_1)^m>(1-t_1)^n$. However this is impossible because $f$ is monotonous increasing.
\end{proof}

\section{Homogeneous topological quandles}\label{sec4}
The main aim of this section is to give a general description of homogeneous topological quandles. The algebraic version of this construction was first given by Joyce in \cite{Joyce}, which also can be found in \cite{Hulpke}. In this section, all topological spaces are assumed to be topological manifolds.

Let $G$ be a topological group and $\sigma$ an automorphism on $G$. As we mentioned in section~\ref{sec2}, $(G, \ast)$ is a topological quandle if we define $x\ast y=\sigma(xy^{-1})y$. If there is a subgroup $H$ of $G$ such that $\sigma(h)=h$ for every $h\in H$, then the right cosets $G/H$ is a topological space and it inherits a topological quandle structure from $G$. Since $G$ acts transitively on the right of $G/H$, we conclude that $(G/H, \ast)$ is a homogeneous topological quandle. The following theorem is a topological version of Theorem 7.1 in \cite{Joyce}.
\begin{theorem}
Every homogeneous topological quandle $X$ can be realized as $(G/H, \ast)$ discussed above.
\end{theorem}
\begin{proof}  Let $X$ be a homogeneous topological quandle.  It suffices to construct a topological group $G$ and a subgroup $H$, then prove that $X$ is isomorphic to $G/H$ as a topological quandle. Define $G=Aut(X)$ and $\sigma$ denotes the conjugation by $R_x$, where $x$ is a fixed point of $X$, that is $\sigma(f)=R_x^{-1}fR_x$. Let $H=\{f\in G|\; f(x)=x\}$.

Since $X$ is a topological manifold, then it is locally compact, Hausdorff and hence Tychonoff. Recall that a topological space is uniformizable if and only if it is Tychonoff. Now we choose the fine uniformity $\mathfrak{U}$ which is compatible with the original topology of $X$. Now for each $U\in \mathfrak{U}$ we define $\mathfrak{V}=\{(f, g)|\; (f(x), g(x))\in U, f,g\in Aut(X), x\in X\}$. It is easy to see that $\mathfrak{V}$ forms a uniform structure on $Aut(X)$. Since $\mathfrak{U}$ is fine, then each $f\in Aut(X)$ is uniformly continuous. It was proved by Ford in \cite{Ford} that if a group consisted of uniformly continuous homeomorphisms of $X$, then it is a topological group relative to the uniform convergence notion induced by the uniform structure of $X$. Now it is evident that $\phi: G/H\rightarrow X$, defined by $\phi(Hf)=f(x)$, is continuous and 1-1. On the other hand, each topological manifold is a SLH (strong local homogeneity) space. According to Theorem 4.1 in \cite{Ford}, if $X$ is a SLH, Tychonoff space and $Aut(X)$ is transitive, then $\phi: G/H\rightarrow X$ is open. This completes the proof that $G/H$ and $X$ are homeomorphism.

Finally
\begin{center}
$\phi(Hf\ast Hg)=\phi(H\sigma(fg^{-1})g)=\phi(HR_x^{-1}fg^{-1}R_xg)=(gR_xg^{-1}f)(x)=f(x)\ast g(x)$.
\end{center}
Hence $\phi$ is a quandle homomorphism. This finishes the proof.
\end{proof}

\section{A set of open questions}
\subsection{Topological Quandle Structures on the Closed Unit Interval}
We begin this section with a positive answer to the second question in the end of section~\ref{sec2}.
\begin{theorem}
There exist infinitely many topological spaces which only can be equipped with the trivial quandle structure.
\end{theorem}
\begin{proof}
Let us consider the topological space $X=\{1, 2, \cdots, n\}$ equipped with the topology $\tau=\{\varnothing, \{1\}, \{1, 2\}, \cdots, \{1, \cdots, n\}\}$. Let $\ast$ be a quandle operation on $X$. Since $\{1\}$ is a open set, then $1\ast^{-1}i$ $(1\leq i\leq n)$ must be a open set, hence equals 1. On the other hand, since $\{1, 2\}$ is a open set and $1\ast^{-1}i=1$ $(1\leq i\leq n)$, we must have $2\ast^{-1}i=2$ $(1\leq i\leq n)$. Eventually we will find that $n\ast^{-1}i=n$ $(1\leq i\leq n)$. The proof is completed.
\end{proof}

Note that the topological space above is not Hausdorff. In general we would like to know that whether we can find a manifold example. We conjecture that the closed unit interval $[0, 1]$ is an example of this kind. In other words, we conjecture that the only topological quandle structure on $[0,1]$ is the trivial operation given by $x*y=x, \forall x, y \in [0,1]$. It obviously follows that $[a, b]$ has no non-trivial quandle structure for any $-\infty<a<b<+\infty$.

\begin{conjecture}\label{conjecture}
There is no non-trivial topological quandle structure on the closed interval $[0,1]$ of the real line.
\end{conjecture}

Geometrically, the conjecture above implies that if $\ast$ is a quandle operation on $[0, 1]$ then the closed surface $F=\{(x, y, x\ast y)\subset\mathbb{R}^3|0\leqslant x\leqslant1, 0\leqslant y\leqslant1\}$ is flat, i.e. $F=\{(x, y, x)\subset\mathbb{R}^3|0\leqslant x\leqslant1, 0\leqslant y\leqslant1\}$. Let us use $\mathcal{F}$ to denote $\{(x, y, x)\subset\mathbb{R}^3|0\leqslant x\leqslant1, 0\leqslant y\leqslant1\}$. According to the definition of quandle, we observe that the segment $\{(x, x, x\ast x)|0\leqslant x\leqslant1\}\subset\mathcal{F}$. Besides of this, we also observe that $\partial F=\partial\mathcal{F}$:
\begin{itemize}
\item $0\ast x=0$ and $1\ast x=1$ for any $x\in [0, 1]$. Actually, since $R_x$ is an automorphism of $[0, 1]$ then it follows that $R_x(0)=0$ or 1. If $R_x(0)=1$ for some $x\in [0, 1]$, recall that $R_0(0)=0$, then there exists $t\in (0, x)$ such that $0<R_t(0)<1$, which means that $R_t$ is not an automorphism of $[0, 1]$. Therefore we always have $0\ast x=0$. It follows immediately that $1\ast x=1$ for any $x\in [0, 1]$.
\item $R_0=R_1=id$. We claim that if $R_0(x)=y$ $($or $R_1(x)=y)$, then $R_x=R_y$. In fact, for any $z\in [0, 1]$ there exists a real number $w\in [0, 1]$ such that $R_0(w)=z$. If $x\ast 0=y$, then we have
\begin{center}
$z\ast x=(w\ast 0)\ast x=(w\ast x)\ast(0\ast x)=(w\ast x)\ast 0=(w\ast 0)\ast(x\ast 0)=z\ast y$.
\end{center}

Now we prove that $R_0=id$. First let us consider a simple case: $x\ast 0\neq x$ for any $x\in (0, 1)$. Without loss of generality, we assume the inequality $x\ast 0>x$ holds for all $x\in (0, 1)$. For a fixed point $x\in (0, 1)$, we have the following monotonically increasing sequence
\begin{center}
$\{\cdots, R_0^{-3}(x), R_0^{-2}(x), R_0^{-1}(x), x, R_0(x), R_0^{2}(x), R_0^{3}(x), \cdots\}$.
\end{center}
For simplicity, we will use $x_n (n\in \mathbb{Z})$ to denote $R_0^n(x)$. In particular, $x_0=x$. According to our discussion above, we know that $R_{x_n}$ are all equivalent. It is clear that $\lim\limits_{n\rightarrow +\infty}x_n$ exists. Actually we must have $\lim\limits_{n\rightarrow +\infty}x_n=1$. Otherwise $\lim\limits_{n\rightarrow +\infty}x_n<1$ and $R_0(\lim\limits_{n\rightarrow +\infty}x_n)=\lim\limits_{n\rightarrow +\infty}R_0(x_n)=\lim\limits_{n\rightarrow +\infty}x_{n+1}=\lim\limits_{n\rightarrow +\infty}x_n$, which contradicts with our assumption that $R_0(x)>x$ for any $0<x<1$. In a similar manner one can proves that $\lim\limits_{n\rightarrow -\infty}x_n=0$. Then it follows that
\begin{center}
$R_0=\cdots=R_{x_{-3}}=R_{x_{-2}}=R_{x_{-1}}=R_{x_{0}}=R_{x_{1}}=R_{x_{2}}=R_{x_{3}}=\cdots=R_1$.
\end{center}
Let $x_0=x$ runs over $(0, 1)$ we obtain that $R_0=R_x=R_1$ for any $x\in (0, 1)$. Then for any $x\in [0, 1]$, we have $x\ast 0=x\ast x=x$, which contradicts with our assumption that $x\ast 0>x$ for any $x\in (0, 1)$. Hence we conclude that $R_0=id$. The result $R_1=id$ can be proved by an analogous argument.

For the general case, if $R_0\neq id$, then there exists some $x_0\in (0, 1)$ such that $R_0(x_0)\neq x_0$. Without loss of generality, we assume that $x_0\ast 0>x_0$. Let $x_+=\min\{t|t>x_0, t\ast0=t\}$ and $x_-=\max\{t|t<x_0, t\ast0=t\}$. Then the inequality $x\ast 0>x$ holds for any $x\in (x_-, x_+)$. Note that $R_0[x_-, x_+]=[x_-, x_+]$, since $R_0$ is monotonically increasing. By repeating the argument above one can prove that $R_x=R_y$ for any $x, y\in [x_-, x_+]$. In particular, it follows that $x\ast y=x\ast x=x$ for any $x, y\in [x_-, x_+]$.
\end{itemize}

In particular, we can prove a special case of Conjecture \ref{conjecture} as follows. If we denote the binary operation in a quandle by a map $f:X \times X \rightarrow X$ sending $(x,y)$ to $f(x,y)$, then the right distributivity axiom of a quandle can be written in the form
\begin{eqnarray}\label{Dist1}
    f(f(x,y),z)=f(f(x,z),f(y,z))
\end{eqnarray}
Now let assume that the function $f$ is a real polynomial in the variables $x$ and $y$.

Then we have the following
\begin{lemma}
Any polynomial solution $P(x,y)\in\mathbb{R}[x,y]$ to the equation (\ref{Dist1}) is either of the form $P(x,y)=ax+(1-a)y$ or $P(x,y)=P(x)$ a polynomial in the variable $x$ only.
\end{lemma}
\begin{proof}
First we fix a few notations.  For any polynomial $f(x,y)\in \mathbb{R}[x,y]$, denote by $f_x$ the highest power of $x$, by $f_y$ the highest power of $y$ and by $f_{xy}$ the degree of the polynomial $f(w,w)$.  For example, for $f(x,y)=2x^3 -x+x^4 y^5$, we have $f_x=3,\; f_y= 0$ and $f_{xy}=9.$

Now the equation $f(f(x,x),y)=f(f(x,y),f(x,y))$ gives $f_y = f_y  f_{xy}$.

If $f_y=0$ then in this case  $f(x,y)=P(x)$ is a polynomial in the variable $x$ only.

If $f_{xy}=1$ then in this case $f(x,y)=ax+by+c$.

Now the equation $f(f(x,y),z)=f(f(x,z),f(y,z))$ implies that
$$a(ax + by + c) + bz + c = a(ax + bz + c) + b(ay + bz + c) + c.$$

Then $bc=0$ or $ab+b^2=b$.  Thus either $b=0$ and $f(x,y)$ is a degree one polynomial in $x$ only, or $b\neq 0$ and thus $c=0$ and $a+b-1=0$ giving $f(x,y)= ax+(1-a)y$.  This concludes the proof.
\end{proof}

Note that if the binary operation $x\ast y=P(x)$ is a polynomial in the variable $x$ only, then the idempotency equation $x\ast x=x$ gives that $P(x)=x$. On the other hand, if $x\ast y=ax+(1-a)y$ is a quandle operation on the closed unit interval, since $0\ast y=0$, it follows that $a=1$. In conclusion, if a binary operation $\ast$ on $[0, 1]$ satisfies $x\ast y=P(x, y)\in\mathbb{R}[x,y]$, then $x\ast y=x$.

\subsection{Some other open questions}
\begin{question}
More generally, First, start by classifying the indecomposable compact connected topological quandles.
\end{question}

\begin{question}
It is well know result \cite{HofmannMostert} that the underlying topology of compact connected abelian group completely determines its structure as a topological group.
That is if $H_1$ and $H_2$ are compact connected abelian groups, then if  $H_1$ and $H_2$ are homeomorphic then $H_1$ and $H_2$ are isomorphic as topological groups.	Find an "analogous" result for topological quandles.  If two compact connected medial quandles are homeomorphic then are they isomorphic as topological quandles?
\end{question}

\begin{question}
For a given topological space $X$, is $X$ a topological rack/quandle with a nontrivial rack/quandle structure? In particular, is there any nontrivial quandle structure on closed orientable surface with genus greater than one?
\end{question}

\begin{question}
	In definition 4.3 of \cite{CS}, Clark and Saito studied a family of quandle structures on the $2$-sphere $S^2$ called \emph{spherical quandles}.  Precisely, for $0<\psi<2\pi$, and for $u, v \in S^2$, define a binary operation on $S^2$ by $u*_{\psi}v$ to be the rotation of $u$ about $v$ by the angle $\psi$. Up to isomorphism, are there any other different quandle structures on $S^2$?

%Is the quandle structure on the $2$-sphere $S^2$ coming from its symmetric space structure, $x * y=2(x \cdot y)y-x,$ where $x\cdot y$ is the usual scalar product in $\mathbb{ R}^{3}$,  the unique (up to isomorphism) non-trivial quandle structure?  More generally, is this true for the spheres which are not H-space ($S^1, S^3$ and $S^7$).
\end{question}

\section*{Acknowledgements} The authors wish to thank Edwin Clark and Vilmos Totik for fruitful conversations. Zhiyun Cheng is supported by NSFC 11771042 and NSFC 11571038.


\begin{thebibliography}{0}
 	
 	\bib{AG}{article}{
 		author={Andruskiewitsch, Nicol\'as},
 		author={Gra\~na, Mat\'\i as},
 		title={From racks to pointed Hopf algebras},
 		journal={Adv. Math.},
 		volume={178},
 		date={2003},
 		number={2},
 		pages={177--243},
 		issn={0001-8708},
 		review={\MR{1994219}},
 	}
 	\bib{Arens}{article}{
 		author={Arens, R.},
 		title={Topologies for homeomorphism groups},
 		journal={Amer. J. Math.},
 		volume={68},
 		date={1946},
 		number={4},
 		pages={593--610},
 		review={\MR{0019916}},
   }

   \bib{CS}{article}{
   	author={Clark, W. Edwin},
   	author={Saito, M.},
   	title={Longitudinal Mapping Knot Invariant for $SU(2)$},
   	journal={ arXiv:1802.08899, 2018},
   	
   }



   \bib{ESZ}{article}{
   	author={Elhamdadi, M.},
   	author={Saito, M.},
   	author={Zappala, E.},
   	title={Continuous cohomology of topological quandles},
   	journal={ preprint	, 2018},
   	
   }

 	\bib{Elhamdadi-Moutuou:Foundations}{article}{
 		author={Elhamdadi, M.},
 		author={Moutuou, E. M.},
 		title={Foundations of topological racks and quandles},
 		journal={J. Knot Theory Ramifications},
 		volume={25},
 		date={2016},
 		number={},
 		pages={1640002 (17 pages)},
 		%issn={},
 		%review={{}},
 		doi={10.1142/S0218216516400022},
 	}
 	

\bib{EN}{book}{
	author={Elhamdadi, M.},
	author={Nelson, S.},
	title={Quandles---an introduction to the algebra of knots},
	series={Student Mathematical Library},
	volume={74},
	publisher={American Mathematical Society, Providence, RI},
	date={2015},
	pages={x+245},
	isbn={978-1-4704-2213-4},
	review={\MR{3379534}},
}

 	\bib{Ford}{article}{
 		author={Ford, L. R.},
 		title={Homeomorphism groups and coset spaces},
 		journal={Trans. Amer. Math. Soc.},
 		volume={77},
 		date={1954},
 		pages={490--497},
 		review={\MR{0066636}},
   }

\bib{HofmannMostert}{book}{
	author={Hofmann, K. H.},
	author={Mostert, P. S.},
	title={Cohomology theories for compact abelian groups},
	note={With an appendix by Eric C. Nummela},
	publisher={Springer-Verlag, New York-Heidelberg; VEB Deutscher Verlag der
		Wissenschaften, Berlin},
	date={1973},
	pages={236},
	review={\MR{0372113 (51 \#8330)}},
}

\bib{Hulpke}{article}{
	author={Hulpke, A.},
	author={Stanovsk\'{y}, D.},
    author={Vojt\v{e}chovsk\'{y}, P.},
	title={Connected quandles and transitive groups},
    journal={J. Pure Appl. Algebra},
    volume={220},
    date={2016},
    number={2},
    pages={735--758},
    issn={0022-4049},
    review={\MR{3399387}},
    %doi={10.1016/j.jpaa.2015.07.014},
}



\bib{Jacobsson-Rubinsztein}{article}{
	author={Jacobsson, M.},
	author={Rubinsztein, R.},
	title={Symplectic Topology of SU(2)-Representation Varieties and Link Homology, I: Symplectic Braid Action and the First Chern Class},
	status={arXiv:0806.2902v3, 2009 },
}
\bib{Joyce}{article}{
	author={Joyce, D.},
	title={A classifying invariant of knots, the knot quandle},
	journal={J. Pure Appl. Algebra},
	volume={23},
	date={1982},
	number={1},
	pages={37--65},
	issn={0022-4049},
	review={\MR {638121 (83m:57007)}},
	doi={10.1016/0022-4049(82)90077-9},
}

\bib{KM}{article}{
	author={P. B. Kronheimer},
    author={T. S. Mrowka},
	title={Knot homology groups from instantons},
	journal={Journal of Topology},
	volume={4},
	date={2011},
    number={4},
	pages={835--918},
	issn={1753-8416},
	review={\MR {2860345}},
	doi={10.1112/jtopol/jtr024},
}

\bib{Matveev}{article}{
	author={Matveev, S. V.},
	title={Distributive groupoids in knot theory},
	language={Russian},
	journal={Mat. Sb. (N.S.)},
	volume={119(161)},
	date={1982},
	number={1},
	pages={78--88, 160},
	issn={0368-8666},
	review={\MR {672410 (84e:57008)}},
}

%\bib{Reem}{article}{
%	author={Daniel Reem},
%	title={Remarks on the Cauchy Functional equation and variations of it},
%	journal={arXiv:1002.3721v7},
%}

\bib{Rubinsztein:Top_Quandles}{article}{
  author={Rubinsztein, R. L.},
  title={Topological quandles and invariants of links},
  journal={J. Knot Theory Ramifications},
  volume={16},
  date={2007},
  number={6},
  pages={789--808},
  issn={0218-2165},
  review={\MR {2341318 (2008e:57012)}},
  doi={10.1142/S0218216507005518},
}
\end{thebibliography}
  \end{document}